\documentclass[12pt]{article}
\usepackage[english]{babel}
\usepackage{subcaption}
\usepackage{cite}
\usepackage{amsmath,amssymb,amsfonts,amsthm}
\usepackage{algorithmic}
\usepackage{mathtools}
\usepackage{graphicx}
\usepackage{subfig}
\usepackage{textcomp}
\usepackage{xcolor}
\usepackage{comment}
\usepackage{array}
\usepackage{tabularx}

\theoremstyle{plain}
\newtheorem{theorem}{Theorem}[section]
\newtheorem{lemma}[theorem]{Lemma}

\theoremstyle{definition}
\newtheorem{example}{Example}[section]
\newtheorem{definition}{Definition}[section]

\newcommand{\trans}{^{\mathrm{T}}}

\def\BibTeX{{\rm B\kern-.05em{\sc i\kern-.025em b}\kern-.08em
    T\kern-.1667em\lower.7ex\hbox{E}\kereln-.125emX}}


\title{On the number of generalized cospectral mates of graphs}

\author{\small Muhammad Raza$^{{\rm a}}$\thanks{Corresponding author: mraza@itu.edu.pk}\quad\quad Obaid Ullah Ahmad$^{\rm b}$\quad\quad Mudassir Shabbir$^{\rm a}$\quad\quad Waseem Abbas$^{{\rm c}}$
\\
{\footnotesize$^{\rm a}$Department of Computer Science, Lahore University of Management Sciences, Lahore, Pakistan}\\
{\footnotesize$^{\rm b}$Department of Electrical Engineering, The University of Texas at Dallas, Richardson, TX}\\
{\footnotesize$^{\rm c}$Department of Systems Engineering, The University of Texas at Dallas, Richardson, TX
}
}

\date{}

\begin{document}

\maketitle

\begin{abstract}

This paper establishes an upper bound on the number of generalized cospectral mates of simple graphs, where the generalized spectrum consists of the spectrum of a graph and its complement. Moving beyond the classical problem of identifying graphs determined by their generalized spectrum, we address the more quantitative question of how many non-isomorphic graphs can share the same generalized spectrum. Our approach is based on arithmetic constraints derived from the Smith Normal Form (SNF) of the walk matrix, which leads to a tight upper bound on the number of generalized cospectral mates of a graph. Our upper bound applies to a much broader class of graphs than those previously shown to have no generalized cospectral mates (graphs determined by generalized spectrum). Consequently, this work extends the family of graphs for which strong and informative spectral uniqueness results are available.

\end{abstract}

\section{Introduction}
\label{section:intro}
Let $G = (V, E)$ be a simple graph with $n$ vertices, meaning it is finite, undirected, and contains no loops or multiple edges. The \emph{adjacency matrix} $A(G)$ is the $n \times n$ symmetric matrix where each entry $a_{ij} = 1$ if vertices $v_i$ and $v_j$ are adjacent, and $a_{ij} = 0$ otherwise. The \emph{spectrum} of a graph $G$, denoted $\operatorname{Spec}(G)$, is the multiset of eigenvalues of its adjacency matrix $A(G)$, counted with algebraic multiplicity.

Two graphs $G$ and $H$ are \emph{isomorphic}, denoted $G \cong H$, if there exists a bijection $f: V(G) \to V(H)$ such that vertices $u$ and $v$ are adjacent in $G$ if and only if $f(u)$ and $f(v)$ are adjacent in $H$. In other words, isomorphic graphs have identical structure, differing only by a relabeling of vertices. The \emph{complement} of a graph $G$, denoted $\overline{G}$, is the graph on the same vertex set where two vertices are adjacent if and only if they are not adjacent in $G$. The adjacency matrix of the complement is given by $A(\overline{G}) = J - I - A(G)$, where $J$ denotes the $n \times n$ matrix of all ones and $I$ is the identity matrix.

A graph $G$ is said to be \emph{determined by its spectrum} (DS) if every graph with the same spectrum as $G$ is isomorphic to $G$. A central open question in spectral graph theory is: “Which graphs are DS?” This question was first raised by Günthard and Primas \cite{gunthard1956zusammenhang} in the context of chemistry and is closely related to Kac’s famous question \cite{kac1966can}, “Can one hear the shape of a drum?” Fisher \cite{fisher1966hearing} modeled the drum as a graph whose sound is described by its spectrum, linking Kac’s question to the graph being DS. Haemers conjectured that almost all graphs are DS \cite{van2003graphs, haemers2016almost}. Nevertheless, proving that a given graph is DS remains challenging; see \cite{van2003graphs, van2009developments} for surveys.

A straightforward extension of this problem is considered in the context of \emph{generalized spectrum}, defined as the pair $(\operatorname{Spec}(G), \operatorname{Spec}(\overline{G}))$ \cite{wang2006sufficient}. This notion incorporates spectral information of both a graph and its complement, often providing stronger structural constraints than the spectrum alone. Two or more graphs are \emph{generalized cospectral} if they have the same generalized spectrum. While isomorphic graphs are clearly generalized cospectral, the converse does not hold in general. A set of pairwise non-isomorphic graphs that are generalized cospectral are referred to as \emph{generalized cospectral mates} of each other. A graph is said to be \emph{determined by its generalized spectrum} (DGS) if it has no generalized cospectral mate.


In recent years, the {\em walk matrix} has emerged as a central tool in studying whether a graph is DGS. For a graph $G$ with $n$ vertices, the \emph{walk matrix} $W(G)$ is the $n \times n$ matrix with columns $e,\, A(G)e,\, \dots,\, A(G)^{n-1}e$, where $e$ is the all-ones column vector of length $n$. A graph $G$ is \emph{controllable} if $W(G)$ is non-singular. Wang \cite{wang2017simple} introduced a large family of controllable graphs and proved that all graphs in this family are DGS. Extending these ideas to oriented graphs, Qiu et al. \cite{qiu2021oriented} introduced a family of oriented graphs determined by their generalized skew spectra.

Beyond determining whether a graph is uniquely specified by its generalized spectrum, a natural quantitative question is: how many generalized cospectral mates can it have? Wang et al. \cite{wang2023graphs} identified a family of graphs where each graph has at most one generalized cospectral mate and proposed an algorithm to find such a mate. Subsequently, they provided an algorithm to find generalized cospectral mates for controllable or {\em almost} controllable graphs \cite{wang2025haemers}. However, the practical feasibility of this algorithm is constrained by the computational cost of factorizing large integers \cite{crandall2005prime}. Pursuing this quantitative direction further, Raza et al. \cite{raza2025upper} and Lin et al. \cite{lin2025upper} independently established upper bounds on the number of generalized cospectral mates for a large family of oriented graphs.

In this paper, we study how structural information encoded in the walk matrix restricts the number of generalized cospectral mates of a graph, with the goal of obtaining explicit upper bounds on this number. Our approach relies on the observation that arithmetic properties of the walk matrix impose strong constraints on the rational orthogonal matrices that relate adjacency matrices of generalized cospectral graphs. In particular, the Smith Normal Form of the walk matrix imposes divisibility conditions that any such matrix must satisfy, thereby limiting the set of admissible mates. Motivated by these observations, we introduce a family of controllable graphs whose walk matrices possess a highly constrained Smith Normal Form.

Let $d_1(M) \mid d_2(M) \mid \dots \mid d_n(M)$ denote the invariant factors of the Smith Normal Form of a matrix $M$. We define $\mathcal{F}_n$ as the family of controllable graphs on $n$ vertices where each graph $G$ satisfies:
\begin{equation}
\label{eqn:Fn}
d_{\lceil n/2 \rceil}(W(G)) = 1 \quad \text{and} \quad d_{n-1}(W(G)) = 2. 
\end{equation}
To analyze generalized cospectral mates of graphs in $\mathcal{F}_n$, we study the arithmetic structure of rational orthogonal matrices that relate them. A key parameter is the {\em level} of such a matrix.

\begin{definition}
The level of a rational matrix $Q$, denoted $\ell(Q)$, is defined as the smallest positive integer $x$ such that $xQ$ is an integral matrix.
\end{definition}

Our first theorem shows that, within $\mathcal{F}_n$, the level uniquely determines any generalized cospectral mate up to isomorphism.

\begin{theorem}
\label{theorem:level-unique}
Let $G \in \mathcal{F}_n$, and let $H_1$ and $H_2$ be graphs that are generalized cospectral to $G$. Let $Q_1$ and $Q_2$ be the rational orthogonal matrices satisfying $Q_i\trans A(G) Q_i = A(H_i)$ and $Q_i e = e$ for $i=1,2$. If $\ell(Q_1) = \ell(Q_2)$, then $H_1 \cong H_2$.
\end{theorem}

The existence of the orthogonal matrices appearing in Theorem~\ref{theorem:level-unique} follows from a known result, which we recall in Section~\ref{section:prelim}. The result implies that each admissible level can correspond to at most one generalized cospectral mate of $G$, therefore it suffices to bound the number of admissible levels. Using divisibility properties of the last invariant factor of the walk matrix, we obtain the following bound.

\begin{theorem}
\label{theorem:main}
Let $G \in \mathcal{F}_n$. Let $d_n(W(G)) = \prod_{j=1}^m p_j^{k_j}$ be the prime factorization of the last invariant factor of the walk matrix $W(G)$. Then $G$ can have at most $\left(\prod_{j=1}^m k_j\right)-1$ generalized cospectral mates.
\end{theorem}

We begin by reviewing essential preliminaries in Section~\ref{section:prelim}. In Section~\ref{section:onq}, we investigate the properties of a specific set of orthogonal matrices; these properties are then used to prove our main results in Section~\ref{section:proofs}. We then provide illustrative examples and numerical analysis in Section~\ref{section:examples} to demonstrate the application of our results. Section~\ref{section:conclusion} concludes the paper with a summary of our findings and potential directions for future research.

\section{Preliminaries}
\label{section:prelim}
In this section, we review the essential results required for the proofs of Theorems \ref{theorem:level-unique} and \ref{theorem:main}. A matrix $Q$ is called \emph{regular} if $Qe=e$ where $e$ is the all-ones column vector. The following theorem characterizes graphs sharing the same generalized spectrum.

\begin{theorem}
\label{theorem:q}
\cite{wang2006sufficient} Let $G$ be a controllable graph and let $H$ be a graph generalized cospectral to $G$. Then $H$ must be controllable and there exists a unique regular orthogonal matrix $Q$ with rational entries such that:
\begin{equation}
\label{eq:q}
Q\trans A(G) Q = A(H),
\end{equation}
Moreover, the matrix $Q$ satisfies $Q=W(G) W(H)^{-1}$.
\end{theorem}

Motivated by this theorem, we introduce the following notation. Let $O_n(\mathbb{Q})$ denote the set of regular orthogonal matrices of order $n$ with rational entries. We define:
\begin{equation*}
\Gamma(G) = \{Q \in O_n(\mathbb{Q}) \mid Q\trans A(G) Q = A(H) \text{ for some graph } H\}.
\end{equation*}
A key tool for analyzing matrices in $\Gamma(G)$ is the level, as defined in Section~\ref{section:intro}. If $Q \in \Gamma(G)$ has $\ell(Q) = 1$, then $Q$ is an integral orthogonal matrix. Since $Q$ is also regular, it must be a permutation matrix, and consequently the graph $H$ satisfying Eq.~\eqref{eq:q} is isomorphic to $G$. Therefore, generalized cospectral mates arise only when $\ell(Q) > 1$.

Our strategy is to show that if two matrices in $\Gamma(G)$ have the same level, they must differ by a permutation matrix. The next result establishes that such matrices yield isomorphic graphs.

\begin{lemma}\cite{raza2025upper, wang2023graphs}
\label{lemma:permutation-matrix}
Let $G$ be a graph, and let $H_1$ and $H_2$ be graphs generalized cospectral to $G$. Let $Q_1, Q_2 \in \Gamma(G)$ be the unique matrices satisfying $Q_i\trans A(G) Q_i = A(H_i)$ for $i=1,2$. If $Q_2 = Q_1 P$, where $P$ is a permutation matrix, then $H_1 \cong H_2$.
\end{lemma}

This allows us to reduce the problem of bounding the number of generalized cospectral mates to bounding the number of admissible levels. The Smith Normal Form (SNF) of the walk matrix is central to our approach. An integral matrix $V$ of order $n$ is called \emph{unimodular} if $\det{V} = \pm 1$. The following well-known theorem establishes the existence of the SNF:

\begin{theorem}
(e.g. \cite{schrijver1998book})
For a full-rank integral matrix $M$, there exist unimodular matrices $V_1$ and $V_2$ such that $M = V_1 S V_2$, where $S = \operatorname{diag}(d_1(M), d_2(M), \ldots, d_n(M))$ is the SNF of $M$, and the invariant factors satisfy $d_i(M) \mid d_{i+1}(M)$ for $i = 1, 2, \ldots, n-1$.
\end{theorem}

The following lemma shows that the last invariant factor of the integral matrices bounds the level of the rational orthogonal matrix relating them.

\begin{lemma}
\label{lemma:level-dn}
\cite{qiu2023smith}
Let $X$ and $Y$ be two non-singular integral matrices such that $QX = Y$, where $Q$ is a rational orthogonal matrix. Then $\ell(Q) \mid \gcd(d_n(X),d_n(Y))$.
\end{lemma}

By Theorem~\ref{theorem:q}, $Q\trans W(G) = W(H)$, so $\ell(Q) \mid \operatorname{gcd}(d_n(W(G)), d_n(W(H)))$, and in particular $\ell(Q) \mid d_n(W(G))$. Another useful property of the SNF concerns the rank of a matrix over finite fields, which plays an important role in analyzing the structure of matrices in $\Gamma(G)$. For a prime $p$, we denote the finite field of order $p$ by $\mathbb{F}_p$, and the rank of an integral matrix $M$ over this field by $\operatorname{rank}_p(M)$. Specifically, $\operatorname{rank}_p(M)$ is equal to the number of invariant factors of $M$ not divisible by $p$.

Recall that the definition of $\mathcal{F}_n$ requires $d_{\lceil n/2 \rceil}(W(G)) = 1$ and $d_{n-1}(W(G)) = 2$. It follows that for any $G \in \mathcal{F}_n$, we have $\operatorname{rank}_p(W(G)) = n-1$ for every odd prime $p$ dividing $d_n(W(G))$, and $\operatorname{rank}_2(W(G)) \ge \lceil \frac{n}{2} \rceil$. The following lemma provides a general upper bound on $\operatorname{rank}_2(W(G))$:

\begin{lemma}{\cite{wang2017simple}}
\label{lemma:wang-rank2}
Let $G$ be a graph of order $n$. Then $\operatorname{rank}_2(W(G)) \leq \left\lceil \frac{n}{2} \right\rceil$.
\end{lemma}

Combining this with the lower bound derived above, we conclude that $\operatorname{rank}_2(W(G))=\left\lceil \frac{n}{2} \right\rceil$ for every $G \in \mathcal{F}_n$. We introduce some notation before stating the next result. For integers $m$, $k$, and $n$ (with $n \neq 0$), we write $m^k \parallel n$ to indicate that $m^k \mid n$ but $m^{k+1} \nmid n$. The unique integer $k$ satisfying this condition is denoted by $\operatorname{ord}_m(n)$.

The following lemma uses the rank constraint to show that not every divisor of $d_n(W(G))$ is attainable as $\ell(Q)$.

\begin{lemma}
\label{lemma:pk-out}
\cite{qiu2023smith}
Let $G$ be a graph, $ Q \in \Gamma(G) $ and $ p $ be an odd prime. If $p^k \parallel \det{W(G)}$ and $\operatorname{rank}_p(W(G)) = n-1$, then $\ell(Q) \mid \frac{d_n(W(G))}{p} $.
\end{lemma}

\section{Arithmetic Constraints on Matrices in $O_n(\mathbb{Q})$}
\label{section:onq}


In this section, we establish several arithmetic properties of matrices in $O_n(\mathbb{Q})$. Since $\Gamma(G) \subseteq O_n(\mathbb{Q})$, these properties apply to all matrices arising from generalized cospectral graphs and are necessary for the proof of Theorem~\ref{theorem:level-unique}. We begin by identifying conditions under which the solution space of a linear system modulo $p^k$ is one-dimensional.

\begin{lemma}
\label{lemma:rank-1}
Let $M$ be a full-rank integral matrix of order $n$. If a prime $p$ and integer $k\ge 1$ satisfy $p^k \mid d_n(M)$ and $\operatorname{rank}_p(M)=n-1$, then every solution to $M v \equiv 0 \pmod{p^k}$ is a scalar multiple of a fixed vector $w \not\equiv 0 \pmod{p}$ over the ring $\mathbb{Z}/p^k\mathbb{Z}$.
\end{lemma}
\begin{proof}
Let $M = V_1 S V_2$ be the SNF of $M$. Since $V_1$ is unimodular, it is invertible over $\mathbb{Z}/p^k\mathbb{Z}$. Therefore, the solutions to $M v \equiv 0 \pmod{p^k}$ are equivalent to the solutions of
$$S V_2 v \equiv 0 \pmod{p^k}.$$
Let $v' = V_2 v$. The system becomes $S v' \equiv 0 \pmod{p^k}$.

Since $\operatorname{rank}_p(M)=n-1$, it follows that $p \nmid d_i(M)$ for $i = 1, 2, \ldots, n-1$. Thus, the first $n-1$ entries of $v'$ must be congruent to zero modulo $p^k$. The last entry of $v'$, however, can take any value in $\mathbb{Z}/p^k\mathbb{Z}$. Therefore, the solutions to $S v' \equiv 0 \pmod{p^k}$ are of the form
\begin{equation*}
v' \equiv y w' \pmod{p^k}.
\end{equation*}
where $y$ is any integer in $\mathbb{Z}/p^k\mathbb{Z}$ and $w' = [0, \dots, 0, 1]\trans$.

Since $V_2$ is unimodular, the inverse $V_2^{-1}$ exists and is also unimodular. Multiplying both sides of the relation $v' = V_2 v$ by $V_2^{-1}$, we obtain $v = V_2^{-1} v'$. Substituting the solution $v' \equiv y w' \pmod{p^k}$, we obtain:
$$v \equiv y V_2^{-1} w' \pmod{p^k}.$$

Define $w = V_2^{-1} w'$. Note that $w$ is the last column of the matrix $V_2^{-1}$. We claim that $w \not\equiv 0 \pmod p$. If $w \equiv 0 \pmod p$, then every entry in the last column of $V_2^{-1}$ would be divisible by $p$. Consequently, the determinant of $V_2^{-1}$ would be divisible by $p$. This contradicts the fact that $\det{V_2^{-1}} = \pm 1$. Hence, $w \not\equiv 0 \pmod p$, and every solution $v$ to $M v \equiv 0 \pmod{p^k}$ is a scalar multiple of $w$ over $\mathbb{Z}/p^k\mathbb{Z}$.
\end{proof}

In the next lemma, we explore the properties of the column vectors of the matrices in $O_n(\mathbb{Q})$.

\begin{lemma}
\label{lemma:q-props}
Let $Q \in O_n(\mathbb{Q})$, $\bar{Q} = \ell(Q)Q$ and $v_1, v_2, \dots, v_n$ be the column vectors of $\bar{Q}$. Let $p$ be a prime and $k$ be a positive integer such that $p^k \parallel \ell(Q)$. Then:
\begin{enumerate}
\item $v_i\trans v_j \equiv 0 \pmod{p^{2k}}$ for all $i, j$;
\item $p^{2k} \parallel v_i\trans v_i$ for all $i$;
\item $p^{k} \parallel v_i\trans e$ for all $i$.
\end{enumerate} 
\end{lemma}

\begin{proof}
Let $\ell = \ell(Q)$. Since $Q$ is orthogonal, $Q\trans Q = I$, and therefore $\bar{Q}\trans \bar{Q} = (\ell Q)\trans (\ell Q) = \ell^2 I.$ As the $(i,j)$-th entry of $\bar Q\trans \bar Q$ is $v_i\trans v_j$, it follows that
$$v_i\trans v_j = \begin{cases}
\ell^2 & \text{if } i = j, \\
0 & \text{if } i \neq j. \end{cases}
$$
Since $p^k \parallel \ell$, we have, $p^{2k} \parallel \ell^2$. It follows that $v_i\trans v_j\equiv0\pmod{p^{2k}}$ for all $i, j$ while $p^{2k} \parallel v_i\trans v_i$ for all $i$. This proves (1) and (2).

To prove (3), since $Q$ is regular and orthogonal, we have $Q\trans e=e$. Multiplying by $\ell$ gives $\bar{Q}\trans e = \ell e.$
The $i^\text{th}$ component of $\bar{Q}\trans e$ is $v_i\trans e = \ell$. Since $p^k \parallel \ell$, we conclude that $p^k \parallel v_i\trans e$.
\end{proof}

The following lemmas identify conditions under which integer vectors associated with matrices in $O_n(\mathbb{Q})$ are orthogonal over $\mathbb{Z}/p^k\mathbb{Z}$.

\begin{lemma}
\label{lemma:uv}
Let $u$ and $v$ be $n$-dimensional integer vectors, and let $p$ be an odd prime and $k$ be a positive integer. If the following conditions hold:
\begin{enumerate}
\item $u \equiv \alpha v \pmod{p^k}$ where $\alpha$ is an integer not divisible by $p$,
\item $u\trans u \equiv v\trans v \equiv 0 \pmod{p^{2k}}$,
\end{enumerate}
then $u\trans v \equiv 0 \pmod{p^{2k}}$.
\end{lemma}

\begin{proof}
From condition (1), we have:
$$u - \alpha v \equiv 0 \pmod{p^{k}},$$
Taking the inner product of both sides with itself gives
$$(u - \alpha v)\trans (u - \alpha v) \equiv 0 \pmod{p^{2k}},$$
which expands to
$$u\trans u - 2 \alpha u\trans v + \alpha^2 v\trans v \equiv 0 \pmod{p^{2k}}.$$
Since $u\trans u$ and $v\trans v$ are both congruent to zero modulo $p^{2k}$, we have
$$-2 \alpha u\trans v \equiv 0 \pmod{p^{2k}}.$$
As $p$ is odd and $\alpha$ is not divisible by $p$, it follows that $2 \alpha$ is invertible in $\mathbb{Z}/p^{2k}\mathbb{Z}$. Hence $u\trans v \equiv 0 \pmod{p^{2k}}.$
\end{proof}

\begin{lemma}
\label{lemma:uv-2}
Let $u$ and $v$ be $n$-dimensional integer vectors, and let $k$ be a positive integer. If the following conditions hold:
\begin{enumerate}
\item $u \equiv \alpha v \pmod{2^k}$ where $\alpha$ is an odd integer,
\item $\operatorname{ord}_2(u\trans u) = \operatorname{ord}_2(v\trans v) = 2k$,
\item $\operatorname{ord}_2(u\trans e) = \operatorname{ord}_2(v\trans e) = k$,
\end{enumerate}
then $u\trans v \equiv 0 \pmod{2^{2k}}$.
\end{lemma}

\begin{proof}
Based upon our assumptions, we have: (i) $u\trans u = 2^{2k} a$, (ii) $v\trans v = 2^{2k} b$. (iii) $u\trans e = 2^{k} c$, (iv) $v\trans e = 2^{k} d$, where $a$, $b$, $c$, and $d$ are odd integers. From condition (1), we have
\begin{equation}
\label{eq:uv-2-1}
u - \alpha v = 2^{k} m,
\end{equation}
for some integer vector $m$. Multiplying each side by $e\trans$ gives
$$e\trans u - \alpha e\trans v = 2^k e\trans m.$$
Substituting the values of $e\trans u$ and $e\trans v$ gives
$$2^k c - \alpha (2^k d) = 2^k e\trans m,$$
which simplifies to
$$c - \alpha d = e\trans m.$$
Note that since $c$ and $\alpha d$ are both odd, their difference $c - \alpha d$ is even. Thus, $e\trans m$ is even. Moreover, we can express $m\trans m$ as
\begin{align*}
m\trans m &= \sum_{i=1}^{n} m_i^2 \\
&\equiv \sum_{i=1}^{n} m_i \\
&\equiv e\trans m \pmod{2},
\end{align*}
where $m_i$ is the $i^\text{th}$ entry of $m$. Since $e\trans m$ is even, it follows that $m\trans m$ is also even. 

Now, taking the inner product of both sides of Eq.~\eqref{eq:uv-2-1} with itself gives:
\begin{align*}
(u - \alpha v)\trans (u - \alpha v) &= (2^{k} m)\trans (2^{k} m) \\
u\trans u - 2 \alpha u\trans v + \alpha^2 v\trans v &= 2^{2k} m\trans m.
\end{align*}
Substituting the values of $u\trans u$ and $v\trans v$ gives:
\begin{align*}
2^{2k} a - 2 \alpha u\trans v + \alpha^2 (2^{2k} b) &= 2^{2k} m\trans m \\
-2 \alpha u\trans v &= 2^{2k} (m\trans m - (a + \alpha^2 b)) \\
-\alpha u\trans v &= 2^{2k-1} (m\trans m - (a + \alpha^2 b)).
\end{align*}
Note that since $a$ and $\alpha^2 b$ are both odd, their sum $a + \alpha^2 b$ is even. Since $m\trans m$ is also even, it follows that $m\trans m - (a + \alpha^2 b)$ is even. Thus, we conclude
$$-\alpha u\trans v \equiv 0 \pmod{2^{2k}}.$$
Since $\alpha$ is odd, it is invertible in $\mathbb{Z}/2^{2k}\mathbb{Z}$. Thus, $u\trans v \equiv 0 \pmod{2^{2k}}$.
\end{proof}

We now show that the common column structure of two matrices in $O_n(\mathbb{Q})$ over $\mathbb{Z}/p^k\mathbb{Z}$ forces $p$ out of the level of their product.

\begin{lemma}
\label{lemma:p-out}
Let $Q_1, Q_2 \in O_n(\mathbb{Q})$ and let $p$ be a prime. Suppose that 
$\operatorname{ord}_p(\ell(Q_i)) = k \geq 1$ for $i=1,2$, and that there exists a vector $w \not\equiv 0 \pmod{p}$ such that every column 
of $\ell(Q_i)Q_i$ is a scalar multiple of $w$ 
over $\mathbb{Z}/p^k\mathbb{Z}$. Then $p \nmid \ell(Q_1\trans Q_2)$.
\end{lemma}

\begin{proof}
Let $\bar{Q}_i=\ell(Q_i)Q_i$ and let $R = \bar{Q}_1\trans \bar{Q}_2$. Since $\bar{Q}_1$ and $\bar{Q}_2$ are integral, $R$ is an integral matrix with entries $r_{ij} = u_i\trans v_j$, where $u_i$ and $v_j$ denote the $i^\text{th}$ and $j^\text{th}$ columns of $\bar{Q}_1$ and $\bar{Q}_2$, respectively. We show that $r_{ij} \equiv 0 \pmod{p^{2k}}$ for all $i, j$. By the definition of level, there exist column vectors $u'$ of $\bar{Q}_1$ and $v'$ of $\bar{Q}_2$ such that $u' \not\equiv 0 \pmod{p}$ and $v' \not\equiv 0 \pmod{p}$. Since both $u'$ and $v'$ are also scalar multiples of $w$ over $\mathbb{Z}/p^k\mathbb{Z}$, there exists an integer $\alpha \not\equiv 0\pmod{p}$ such that:
\begin{equation}
\label{eq:u'-v'}
u' \equiv \alpha v' \pmod{p^{k}},
\end{equation}
Furthermore, since every column vector $u_i$ and $v_j$ is a scalar multiple of $w$ over $\mathbb{Z}/p^k\mathbb{Z}$, we can express them as:
\begin{equation}
\label{eq:ui}
u_i \equiv c_i u' + p^{k}x_i \pmod{p^{2k}},
\end{equation}
\begin{equation}
\label{eq:vi}
v_j \equiv d_j v' + p^{k}y_j \pmod{p^{2k}},
\end{equation}
where $c_i, d_j$ are integers and $x_i, y_j$ are integer vectors.

Since $Q_1$ and $Q_2$ satisfy the conditions of Lemma~\ref{lemma:q-props}, it follows that $u_i\trans u' \equiv 0 \pmod{p^{2k}}$ and $v_j\trans v' \equiv 0\pmod{p^{2k}}$ for each column $u_i$ and $v_j$. Moreover, $\operatorname{ord}_p({u'}\trans u') = \operatorname{ord}_p({v'}\trans v') = 2k$ and $\operatorname{ord}_p({u'}\trans e) = \operatorname{ord}_p({v'}\trans e) = k$. Therefore, we can apply Lemma~\ref{lemma:uv} (for odd primes) or Lemma~\ref{lemma:uv-2} (for $p=2$) to establish that ${u'}\trans v' \equiv 0 \pmod{p^{2k}}$.

Multiplying Eq.~\eqref{eq:ui} by ${u'}\trans$ yields:
\begin{align*}
{u'}\trans u_i &\equiv {u'}\trans (c_i u' + p^{k}x_i) \pmod{p^{2k}} \\
&\equiv c_i {u'}\trans u' + p^{k}{u'}\trans x_i \pmod{p^{2k}}
\end{align*}
As ${u'}\trans u_i \equiv 0 \pmod{p^{2k}}$ and $c_i {u'}\trans u' \equiv 0 \pmod{p^{2k}}$, it follows that $p^{k} {u'}\trans x_i \equiv 0 \pmod{p^{2k}}$, which implies ${u'}\trans x_i \equiv 0 \pmod{p^{k}}$. Using Eq.~\eqref{eq:u'-v'}, we get ${u'}\trans x_i \equiv \alpha {v'}\trans x_i \pmod{p^{k}}$. As $\alpha \not\equiv 0 \pmod{p}$, we conclude that ${v'}\trans x_i \equiv 0 \pmod{p^{k}}$. Using a similar approach, we can prove that ${u'}\trans y_j \equiv 0 \pmod{p^{k}}$ for each $y_j$.

Finally, for any columns $u_i$ and $v_j$, we have:
\begin{align*}
u_i\trans v_j &\equiv (c_i u' + p^{k} x_i)\trans (d_j v' + p^{k} y_j) \\
&\equiv c_i d_j {u'}\trans v' + c_i p^{k} {u'}\trans y_j + d_j p^{k} x_i\trans v' + p^{2k} x_i\trans y_j \pmod{p^{2k}}.
\end{align*}
We established that ${u'}\trans y_j \equiv 0 \pmod{p^{k}}$ and ${v'}\trans x_i \equiv 0 \pmod{p^{k}}$. Since ${u'}\trans v' \equiv 0 \pmod{p^{2k}}$, every term on the right-hand side vanishes modulo $p^{2k}$. Thus, we conclude that $u_i\trans v_j \equiv 0 \pmod{p^{2k}}$ for all $i, j$, which implies $R \equiv 0 \pmod{p^{2k}}$. Substituting the definition of $R$, we have
$$\ell(Q_1)\ell(Q_2) Q_1\trans Q_2 \equiv 0 \pmod{p^{2k}}.$$
This implies that  $\frac{\ell(Q_1)\ell(Q_2)}{p^{2k}} Q_1\trans Q_2$ is an integral matrix, and hence $\ell(Q_1\trans Q_2) \mid \frac{\ell(Q_1)\ell(Q_2)}{p^{2k}}$. Since $p^k \parallel \ell(Q_1)$ and $p^k \parallel \ell(Q_2)$, it follows that $p^{2k} \parallel \ell(Q_1)\ell(Q_2)$. Consequently, $p \nmid \frac{\ell(Q_1)\ell(Q_2)}{p^{2k}}$, which implies $p \nmid \ell(Q_1\trans Q_2)$, completing the proof.

\end{proof}

Building on this, we show that if two matrices in $O_n(\mathbb{Q})$ share the same level and satisfy certain column conditions, then their product is a permutation matrix.

\begin{lemma}
\label{lemma:same-level}
Let $Q_1, Q_2 \in O_n(\mathbb{Q})$ with $\ell(Q_1)=\ell(Q_2)=\ell$, and let $\bar Q_i=\ell Q_i$ for $i=1,2$. Let
$$\ell=\prod_{j=1}^m p_j^{k_j}$$
be the prime factorization of $\ell$. Suppose that for each $j \in \{1, \dots, m\}$, there exists a vector $w_j \not\equiv 0\pmod {p_j}$ such that every column of both $\bar Q_1$ and $\bar Q_2$ is a scalar multiple of $w_j$ over $\mathbb Z/p_j^{k_j}\mathbb Z$. Then $Q_1\trans Q_2$ is a permutation matrix.
\end{lemma}
\begin{proof}
Since $\bar{Q}_1$ and $\bar{Q}_2$ are integral, the product $\bar{Q}_1\trans \bar{Q}_2=\ell^2 Q_1\trans Q_2$ is an integral matrix. It follows from the definition of level that $\ell(Q_1\trans Q_2)$ divides $\ell^2$. For each prime factor $p_j$ of $\ell$, the assumptions of Lemma~\ref{lemma:p-out} are satisfied with $k=k_j$. Thus, $p_j \nmid \ell(Q_1\trans Q_2)$. Since $\ell(Q_1\trans Q_2)$ divides $\ell^2$ but contains no prime factor of $\ell$, it must be that $\ell(Q_1\trans Q_2) = 1$. Therefore, $Q_1\trans Q_2$ is an integral orthogonal matrix.

Finally, since $Q_1, Q_2 \in O_n(\mathbb{Q})$, it follows that $Q_1\trans Q_2 \in O_n(\mathbb{Q})$. An integral orthogonal matrix with row sums equal to 1 must be non-negative, which implies $Q_1\trans Q_2$ is a permutation matrix.
\end{proof}

In the next section, we apply these results to prove our main theorems.

\section{Proofs of Theorem \ref{theorem:level-unique}  and Theorem \ref{theorem:main}}
\label{section:proofs}

We begin by defining some auxiliary matrices that are critical for the proof of Theorem~\ref{theorem:level-unique}.

\begin{lemma}\cite{wang2017simple}
\label{lemma:wang-mod2}
Let $\phi(x) = x^n + c_1 x^{n-1} + \cdots + c_{n-1} x + c_n$ be the characteristic polynomial of a graph $G$ with adjacency matrix $A=A(G)$. Define
$$
M(G) :=
\begin{cases}
A^{n/2} + c_2 A^{n/2-1} + \cdots + c_{n-2} A + c_n I, & \text{if $n$ is even}, \\
A^{(n+1)/2} + c_2 A^{(n-1)/2} + \cdots + c_{n-3} A^2 + c_{n-1} A, & \text{if $n$ is odd}.
\end{cases}
$$
Then $M(G) e \equiv 0 \pmod{2}$, where $e$ is the all-ones vector.
\end{lemma}

Using this result, the integral matrix $\hat{W}(G)$ is defined as follows:

\begin{definition} \cite{qiu2023smith}
Let $G$ be a graph with the adjacency matrix $A=A(G)$ and the matrix $M=M(G)$. $\hat{W}(G)$ is defined as:
$$
\hat{W}(G) := \left[\, e,\, Ae,\, \ldots,\, A^{\lceil n/2 \rceil}e,\,
\frac{Me}{2},\,
\frac{AMe}{2},\, \ldots,\, \frac{A^{n-\lceil n/2 \rceil-1}Me}{2}
\,\right].
$$

\end{definition}

As observed by Qiu et al.~\cite{qiu2023smith}, if $H$ is generalized cospectral to $G$ and $Q \in \Gamma(G)$ such that $Q\trans A(G) Q = A(H)$, then $Q\trans M(G) e = M(H) e$ and consequently, $Q\trans \hat{W}(G)=\hat{W}(H)$. The following lemma relates the SNF of $\hat{W}(G)$ to that of the original walk matrix $W(G)$.

\begin{lemma}
\label{lemma:what}
\cite{qiu2023smith}
Let $G$ be a graph with walk matrix $W(G)$. If $\operatorname{rank}_2{W(G)} = \lceil \frac{n}{2} \rceil$, then for each $1 \leq i \leq n$, the invariant factors of $\hat{W}(G)$ satisfy:
$$
d_i(\hat{W}(G)) = 
\begin{cases}
d_i(W(G)) & \text{if } 1 \leq i \leq \left\lceil \frac{n}{2} \right\rceil, \\
\frac{d_i(W(G))}{2} & \text{if } \left\lceil \frac{n}{2} \right\rceil < i \leq n.
\end{cases}
$$
\end{lemma}

The next lemma establishes the $\operatorname{rank}_p(\hat{W}(G))$ for every prime factor $p$ of $d_n(\hat{W}(G))$ when $G \in \mathcal{F}_n$.

\begin{lemma}
\label{lemma:rank}
Let $G \in \mathcal{F}_n$, then for every prime $p$ dividing $d_n(\hat{W}(G))$, the $\operatorname{rank}_p(\hat{W}(G)) = n-1$.
\end{lemma}

\begin{proof}
As $G\in\mathcal{F}_n$, we established earlier that $\operatorname{rank}_2(W(G))=\lceil \frac{n}{2} \rceil$. Therefore, by Lemma~\ref{lemma:what}, the invariant factors of $\hat{W}(G)$ are given by:
$$d_i(\hat{W}(G)) =
\begin{cases}
1 & \text{if } 1 \leq i \leq n-1, \\
\frac{d_n(W(G))}{2} & \text{if } i=n.
\end{cases}$$
It follows that $\operatorname{rank}_p(\hat{W}(G))=n-1$ for every prime $p$ dividing $d_n(\hat{W}(G))$. This completes the proof.
\end{proof}

We can now proceed to the proof of Theorem~\ref{theorem:level-unique}.

\begin{proof}[Proof of Theorem~\ref{theorem:level-unique}]
Let the prime factorization of $\ell$ be given by

$$\ell = \prod_{j=1}^m p_j^{k_j}.$$
Let $\bar{Q}_i = \ell Q_i$ be the associated integral matrices. Since $Q_i\trans \hat{W}(G) = \hat{W}(H_i)$, by Lemma~\ref{lemma:level-dn}, $\ell(Q_i) \mid d_n(\hat{W}(G))$ and therefore $p_j^{k_j} \mid d_n(\hat{W}(G))$. Combined with Lemma~\ref{lemma:rank}, this gives $\operatorname{rank}_{p_j}(\hat{W}(G))=n-1$.

Additionally, transposing and multiplying by $\ell$ gives $\hat{W}(G)\trans \bar{Q}_i = \ell \hat{W}(H_i)\trans$. As $\hat{W}(H_i)$ is an integral matrix and $\ell\equiv0\pmod{p_j^{k_j}}$ for each prime $p_j$, it follows that $\hat{W}(G)\trans \bar{Q}_i \equiv 0 \pmod{p_j^{k_j}}.$ Consequently, for each $p_j$, the columns of $\bar{Q}_i$ are the solutions to a linear system satisfying the conditions of Lemma~\ref{lemma:rank-1}. Hence, there exists a vector $w_j \not\equiv 0 \pmod{p_j}$ such that the columns of both $\bar{Q}_1$ and $\bar{Q}_2$ are scalar multiples of $w_j$ over the ring $\mathbb{Z}/p_j^{k_j}\mathbb{Z}$.

The conditions of Lemma~\ref{lemma:same-level} are therefore satisfied, implying that $Q_1\trans Q_2 = P$ for some permutation matrix $P$. It follows that $Q_2 = Q_1 P$. By Lemma~\ref{lemma:permutation-matrix}, we conclude that $H_1$ is isomorphic to $H_2$.
\end{proof}

The proof of Theorem~\ref{theorem:main} follows as a direct consequence of Theorem~\ref{theorem:level-unique} by counting the possible levels of matrices in $\Gamma(G)$.

\begin{proof}[Proof of Theorem~\ref{theorem:main}]
As 2 must be a factor of $d_n(W(G))$, let the prime factorization of $d_n(W(G))$ be written as $2^{k_1} \prod_{j=2}^m p_j^{k_j}$, where $p_1=2$ and $p_2, \dots, p_m$ are odd primes. For any matrix $Q \in \Gamma(G)$, Lemma~\ref{lemma:level-dn} gives $\ell(Q) \mid d_n(\hat{W}(G))$. By Lemma \ref{lemma:what}, $d_n(\hat{W}(G))=\frac{d_n(W(G))}{2}$, so the multiplicity of 2 in $d_n(\hat{W}(G))$ is $k_1 - 1$, which implies $2^{k_1} \nmid \ell(Q)$.

Moreover, for each odd prime $p_j$ ($j \ge 2$), by Lemma~\ref{lemma:pk-out}, we know that $\ell(Q) \mid \frac{d_n(W(G))}{p_j}$. It follows that for each prime $p_j$, the exponent satisfies $\operatorname{ord}_{p_j}(\ell(Q)) \in \{0, 1, \dots, k_j-1\}$. Thus, there are exactly $k_j$ possible choices for the exponent of each prime factor. Let $L(G)$ denote the set of all possible values for $\ell(Q)$. The total number of these values is the product of the number of choices for each prime factor:
$$
|L(G)| = \prod_{j=1}^m k_j.
$$

We must exclude the trivial case of the level $\ell(Q)=1$, which corresponds to graphs isomorphic to $G$. By Theorem~\ref{theorem:level-unique}, each valid level corresponds to at most one generalized cospectral mate of $G$. Therefore, the number of generalized cospectral mates of $G$ is at most 
$$|L(G)| - 1 = \left(\prod_{j=1}^m k_j\right) - 1.$$
This completes the proof.
\end{proof}

\section{An Example and Numerical Evaluation}
\label{section:examples}

\begin{figure}[!t]
\centering
\begin{subfigure}[b]{0.49\linewidth}
\hspace{-10mm}\centering
\includegraphics[scale=1]{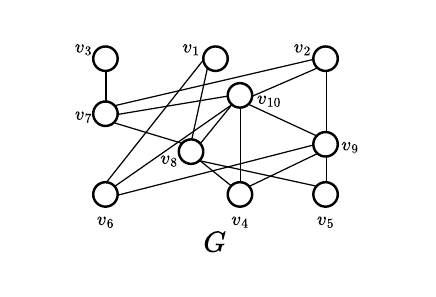}
\end{subfigure}
\begin{subfigure}[b]{0.49\linewidth}
\hspace{-10mm}\centering
\includegraphics[scale=1]{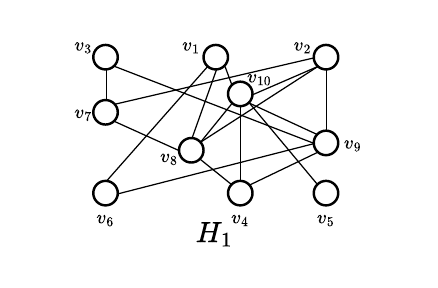}
\end{subfigure}
\begin{subfigure}[b]{0.49\linewidth}
\hspace{-10mm}\centering
\includegraphics[scale=1]{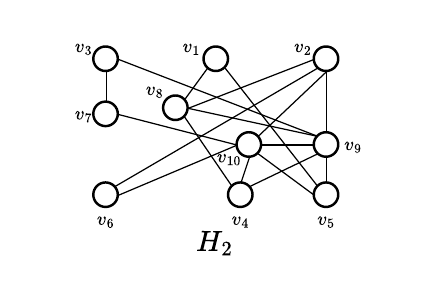}
\end{subfigure}
\begin{subfigure}[b]{0.49\linewidth}
\hspace{-10mm}\centering
\includegraphics[scale=1]{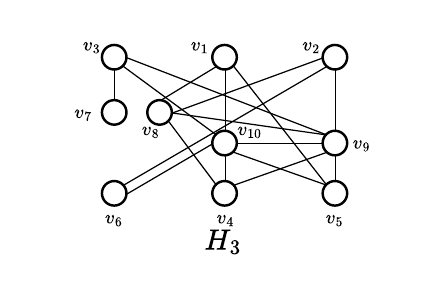}
\end{subfigure}
\caption{Graph $G$ and all of its generalized cospectral mates $H_1$, $H_2$ and $H_3$.}
\label{fig:example_tight}
\end{figure}

In this section, we present an example of a graph in $\mathcal{F}_n$ alongside experimental results.

\begin{example}
Let $n=10$. Consider the graph $G$ shown in Fig.~\ref{fig:example_tight}, with the adjacency matrix:
$$A(G)=
\begin{pmatrix}
0&0&0&0&0&1&0&1&0&0\\
0&0&0&0&0&0&1&0&1&1\\
0&0&0&0&0&0&1&0&0&0\\
0&0&0&0&0&0&0&1&1&1\\
0&0&0&0&0&0&0&1&1&0\\
1&0&0&0&0&0&0&0&1&1\\
0&1&1&0&0&0&0&1&0&1\\
1&0&0&1&1&0&1&0&0&1\\
0&1&0&1&1&1&0&0&0&1\\
0&1&0&1&0&1&1&1&1&0
\end{pmatrix}$$
The Smith Normal Form of the walk matrix $W(G)$ has invariant factors given by:
$$\operatorname{diag}(1, 1, 1, 1, 1, 2, 2, 2, 2, 19350)
.$$
As these factors satisfy Eq.~\eqref{eqn:Fn}, it follows that $G\in \mathcal{F}_n$. The prime factorization of the last invariant factor $d_n(W(G))$ is:

$$d_n(W(G)) = 2 \times 3^2 \times 5^2 \times 43.$$

According to Theorem~\ref{theorem:main}, $G$ can have at most $3$ generalized cospectral mates. By a brute-force search over all graphs of order $10$, we find that $G$ has exactly $3$ generalized cospectral mates, denoted $H_1, H_2,$ and $H_3$ (see Fig.~\ref{fig:example_tight}). Thus, $G$ is a tight example attaining the bound given in Theorem~\ref{theorem:main}.

Let $Q_1, Q_2, Q_3 \in \Gamma(G)$ be the orthogonal matrices satisfying $A(H_i) = Q_i\trans A(G) Q_i$ for $i=1,2,3$. Using Theorem~\ref{theorem:q}, we compute these matrices and find that their levels are $\ell(Q_1)=3$, $\ell(Q_2)=15$, and $\ell(Q_3)=5$, confirming that the three generalized cospectral mates correspond to pairwise distinct levels, in accordance with Theorem~\ref{theorem:level-unique}.
%

\end{example}

%

\begin{table}[htbp]
\centering
\begin{tabular}{c c c}
\hline
$n$ & \# controllable graphs & \# graphs in $\mathcal{F}_n$ \\
\hline
11 & 6,705 & 3,573 \\
12 & 7,708 & 3,646 \\
13 & 8,441 & 3,844 \\
14 & 9,061 & 3,895 \\
15 & 9,427 & 3,911 \\
20 & 9,982 & 3,979 \\
25 & 9,998 & 3,897 \\
30 & 10,000 & 3,973 \\
35 & 10,000 & 3,930 \\
40 & 10,000 & 3,953 \\
45 & 10,000 & 3,908 \\
50 & 10,000 & 3,920 \\
\hline
\end{tabular}

\caption{Statistics for 10,000 random graphs ($p=0.5$) for each order $n$. The columns report, respectively, the order $n$, the number of controllable graphs and the number of graphs in $\mathcal{F}_n$.}
\label{tab:snf-summary}
\end{table}

To assess how broadly Theorem~\ref{theorem:main} applies, we conducted a numerical experiment in which we generated 10,000 random graphs with edge probability $0.5$ for each graph order $n$. For each $n$, we counted the number of controllable graphs and the number of graphs belonging to $\mathcal{F}_n$. The results, summarized in Table~\ref{tab:snf-summary}, indicate that as $n$ increases, the proportion of graphs belonging to $\mathcal{F}_n$ stabilizes around 39\%.



\section{Conclusion}
\label{section:conclusion}
In this paper, we have addressed the problem of determining the number of generalized cospectral mates for a graph by analyzing the arithmetic properties of the walk matrix. We identified a significant family of controllable graphs, denoted by $\mathcal{F}_n$, and established a concrete upper bound on the number of their generalized cospectral mates. 

By leveraging the Smith Normal Form of the walk matrix and the level of rational orthogonal matrices, we proved that each admissible level corresponds to at most one generalized cospectral mate up to isomorphism, allowing us to bound the number of such mates based on the prime factorization of the last invariant factor of $W(G)$. Our numerical experiments demonstrate that $\mathcal{F}_n$ encompasses a substantial proportion of random graphs, approaching approximately $39\%$ as $n$ increases, and we provided a tight example where the derived bound is attained.

Several avenues for future research emerge from this work. First, the definition of $\mathcal{F}_n$ relies on specific constraints on the invariant factors, namely that $d_{\lceil n/2 \rceil} = 1$ and $d_{n-1} = 2$. A natural direction would be to relax these conditions to cover a broader range of controllable graphs. Generalizing the modular constraints on the rational orthogonal matrices could extend the applicability of these bounds to families where the rank of the walk matrix over finite fields behaves differently.

Furthermore, the algebraic techniques developed here need not be limited to the adjacency matrix. Future investigations could explore establishing similar upper bounds in other settings, such as using the signless Laplacian or the distance matrix. Finally, while this study focused on the walk matrix generated by the all-ones vector $e$, employing different control input vectors could reveal new structural insights. This approach might facilitate the determination of generalized cospectrality for graphs that are not controllable with respect to $e$ but may be controllable with respect to other vectors.


\begin{thebibliography}{00}

\bibitem{crandall2005prime} R. Crandall, C. Pomerance, Prime Numbers: A Computational Perspective, 2nd ed., Springer, New York, 2005.

\bibitem{fisher1966hearing} M. E. Fisher, On hearing the shape of a drum, Journal of Combinatorial Theory 1 (1) (1966) 105–125.

\bibitem{gunthard1956zusammenhang} H. H. G{\"u}nthard, H. Primas, Zusammenhang von graphentheorie und mo-theorie von molekeln mit systemen konjugierter bindungen, Helvetica Chimica Acta 39 (6) (1956) 1645–1653.

\bibitem{haemers2016almost} W. H. Haemers, Are almost all graphs determined by their spectrum, Not. S. Afr. Math. Soc. 47 (1) (2016) 42–45. 

\bibitem{kac1966can} M. Kac, Can one hear the shape of a drum?, The american mathematical monthly 73 (4P2) (1966) 1–23.

\bibitem{lin2025upper} L. Lin, W. Wang, H. Zhang, An upper bound on the number of generalized cospectral mates of oriented graphs, arXiv preprint arXiv:2504.18079 (2025).

\bibitem{qiu2021oriented} L. Qiu, W. Wang, W. Wang, Oriented graphs determined by their generalized skew spectrum, Linear Algebra and its Applications 622 (2021) 316–332.

\bibitem{qiu2023smith} L. Qiu, W. Wang, W. Wang, H. Zhang, Smith normal form and the generalized spectral characterization of graphs, Discrete Mathematics 346 (1) (2023) 113177.

\bibitem{raza2025upper} M. Raza, O. U. Ahmad, M. Shabbir, X. Koutsoukos, W. Abbas, An upper bound on generalized cospectral mates of oriented graphs using skew-walk matrices, arXiv preprint arXiv:2504.17278 (2025).

\bibitem{schrijver1998book} A. Schrijver, The Theory of Linear and Integer Programming, John Wiley \& Sons, 1998.

\bibitem{van2003graphs} E. R. van Dam, W. H. Haemers, Which graphs are determined by their spectrum?, Linear Algebra and its applications 373 (2003) 241–272.

\bibitem{van2009developments} E. R. van Dam, W. H. Haemers, Developments on spectral characterizations of graphs, Discrete Mathematics 309 (3) (2009) 576–586.

\bibitem{wang2006sufficient} W. Wang, C. X. Xu, A sufficient condition for a family of graphs being determined by their generalized spectra, European Journal of Combinatorics 27 (6) (2006) 826–840.

\bibitem{wang2017simple} W. Wang, A simple arithmetic criterion for graphs being determined by their generalized spectra, Journal of Combinatorial Theory, Series B 122 (2017) 438–451.

\bibitem{wang2023graphs} W. Wang, W. Wang, T. Yu, Graphs with at most one generalized cospectral mate, The Electronic Journal of Combinatorics (2023) P1–38.

\bibitem{wang2025haemers} W. Wang, W. Wang, Haemers’ conjecture: an algorithmic perspective, Experimental Mathematics 34 (2) (2025) 147–161.

\end{thebibliography}

\end{document}